\documentclass[a4paper]{amsart}

\PII{}
\makeatletter
\def\@setcopyright{\@empty}
\makeatother

\newcommand{\prn}[1]{\left(#1\right)}
\newcommand{\brc}[1]{\left\{#1\right\}}
\newcommand{\allp}{1\le p\le\infty}
\newcommand{\Lp}{L_{p,\alpha}}
\newcommand{\Lmu}{L_{1,2}}
\newcommand{\norm}[1]{\left\|#1\right\|_{p,\alpha}}
\newcommand{\normpar}[2]{\left\|#1\right\|_{#2}}
\newcommand{\E}{E_n(f)_{p,\alpha}}
\newcommand{\Epar}[2]{E_{#1}\left(#2\right)_{p,\alpha}}
\newcommand{\T}[3]{T_{#1}^{#2}\left(#3\right)}
\newcommand{\hatT}[3]{\hat T_{#1}^{#2}\left(#3\right)}
\newcommand{\Si}[1]{\left(1-#1^2\right)}
\newcommand{\Dl}[3]{\Delta_{#1}^{#2}\left(#3\right)}
\newcommand{\arr}[2]{{{#1}_1,\dots,{#1}_{#2}}}
\newcommand{\w}{\hat\omega_r(f,\delta)_{p,\alpha}}
\newcommand{\wpar}[2]{\hat\omega_{#1}\left(#2\right)_{p,\alpha}}
\newcommand{\Px}[1]{P_{#1}^{(2,2)}}
\newcommand{\Py}[1]{P_{#1+2}^{(0,0)}}
\newcommand{\krn}[1]{%
  \left(
    \frac{\sin\frac{m#1}2}{\sin\frac{#1}2}
  \right)^{2q+4}}
\newcommand{\numericset}[1]{\mathbb #1}
\newcommand{\numN}{\numericset N}

\newtheorem{thm}{Theorem}[subsection]
\newtheorem{lmm}{Lemma}[subsection]
\newtheorem{cor}{Corollary}[subsection]

\newcounter{const}

\numberwithin{const}{thm}
\numberwithin{const}{lmm}
\numberwithin{const}{cor}

\makeatletter
\newcommand{\Cn}[1][]{%
  \stepcounter{const}C_{\theconst}%
  \@ifnotempty{#1}{\newcounter{#1}\setcounter{#1}{\arabic{const}}}}
\makeatother
\newcommand{\refC}[1]{C_{\arabic{#1}}}
\newcommand{\lastC}{C_{\theconst}}
\newcommand{\prevC}[1][1]{%
	{\countdef\n=255
	 \n=\theconst
	 \advance\n by-#1
	 C_{\number\n}}}

\numberwithin{equation}{subsection}

\renewcommand{\theconst}{\arabic{const}}

\DeclareMathOperator*\esssup{ess\ sup}

\begin{document}

\title[On coincidence of classes of functions\dots]
	{On coincidence of classes of functions
		defined by a generalised modulus of smoothness
		and the appropriate inverse theorem}
\author{Faton M.~Berisha}
\address{F.~M.\ Berisha\\
	Faculty of Mathematics and Sciences\\
	University of Prishtina\\
	N\"ena Terez\"e~5\\
	10000 Prishtin\"e\\
	Kosovo}
	\email{faton.berisha@uni-pr.edu}

\date{}

\begin{abstract}
	We give the theorem of coincidence of a class of functions
	defined by a generalised modulus of smoothness
	with a class of functions
	defined by the order of the best approximation
	by algebraic polynomials.
	We also prove the appropriate inverse theorem
	in approximation theory.
\end{abstract}

\maketitle

\setcounter{subsection}{-1}\subsection{}

In~\cite{potapov:mat-99}, an asymmetric operator
of generalised translation was introduced,
by means of it the generalised modulus of continuity was defined,
and the theorem of coincidence of a class of functions
defined by that modulus
with a class of functions with given order of the best approximation
by algebraic polynomials was proved.

In our paper the analogous results are obtained
for a generalised modulus of smoothness of order~$r$.
In addition,
in the present paper we prove a theorem inverse to the Jackson's theorem
related to that modulus of smoothness.

\subsection{}

By~$L_p$ we denote the set of functions~$f$ measurable
on the segment~$[-1,1]$ such that for $1\le p<\infty$
\begin{displaymath}
	\normpar f p
	=\prn{\int_{-1}^1|f(x)|^p\,dx}^{1/p}<\infty,
\end{displaymath}
and for $p=\infty$
\begin{displaymath}
	\normpar f\infty
	=\esssup_{-1\le x\le1}|f(x)|<\infty.
\end{displaymath}

Denote by~$\Lp$ the set of functions~$f$ such that
$f(x)\*(1-x^2)^\alpha\in L_p$,
and put
\begin{displaymath}
	\norm f=\normpar{f(x)(1-x^2)^\alpha}p.
\end{displaymath}

By~$\E$ we denote the best approximation of
the function $f\in\Lp$ by algebraic
polynomials of degree not greater than~$n-1$,
in~$\Lp$ metrics, i.e.
\begin{displaymath}
	\E=\inf_{P_n\in\mathbb P_n}\norm{f-P_n},
\end{displaymath}
where~$\mathbb P_n$ is the set of algebraic polynomials
of degree not greater than~$n-1$.

By $E(p,\alpha,\lambda)$ we denote the class of
functions $f\in\Lp$ satisfying the condition
\begin{displaymath}
	\E\le Cn^{-\lambda},
\end{displaymath}
where $\lambda>0$ and~$C$ is a constant
not depending on~$n$.

For a function~$f$ we define the operator of generalised
translation $\hatT t{}{f,x}$ by
\begin{multline*}
	\hatT t{}{f,x}
	  =\frac1{\pi\Si x}\int_0^\pi
		  \bigg(
			1-\prn{x\cos t-\sqrt{1-x^2}\sin t\cos\varphi}^2\\
	-2\sin^2t\sin^2\varphi+4\Si{x}\sin^2t\sin^4\varphi
		  \bigg)\\
	\times f(x\cos t-\sqrt{1-x^2}\sin t\cos\varphi)\,d\varphi.
\end{multline*}

By means of that operator of generalised translation
we define the generalised difference of order~$r$ by
\begin{align*}
	\Dl t1{f,x}
	  &=\Dl t{}{f,x}=\hatT t{}{f,x}-f(x),\\
	\Dl{\arr t r}r{f,x}
	  &=\Dl{t_r}{}{\Dl{\arr t{r-1}}{r-1}{f,x},x}
	\quad(r=2,3,\dotsc),
\end{align*}
and the generalised modulus of smoothness of order~$r$ by
\begin{displaymath}
	\w=\sup_{\substack{|t_j|\le\delta\\ j=1,2,\dots,r}}
		  \norm{\Dl{\arr t r}r{f,x}}
	\quad(r=1,2,\dotsc).
\end{displaymath}

Denote by $H(p,\alpha,r,\lambda)$ the class
of functions $f\in\Lp$ satisfying the condition
\begin{displaymath}
	\w\le C\delta^\lambda,
\end{displaymath}
where $\lambda>0$ and~$C$ is a constant
not depending on~$\delta$.

\subsection{}

Put $y=\cos t$, $z=\cos\varphi$ in the
operator $\hatT t{}{f,x}$, we denote it
by $\T y{}{f,x}$ and rewrite it in the form
\begin{multline*}
	\T y{}{f,x}
	  =\frac1{\pi\Si x}\int_{-1}^1
		  \big(
			1-R^2-2\Si y\Si z\\
	+4\Si x\Si y\Si{z}^2
		  \big)
		  f(R)\frac{dz}{\sqrt{1-z^2}},
\end{multline*}
where $R=xy-z\sqrt{1-x^2}\sqrt{1-y^2}$.

We define the operator of generalised translation
of order~$r$ by
\begin{align*}
	\T y1{f,x}         &=\T y{}{f,x},\\
	\T{\arr y r}r{f,x} &=\T{y_r}{}{\T{\arr y{r-1}}{r-1}{f,x},x}
	  \quad(r=2,3,\dotsc).
\end{align*}

By $P_\nu^{(\alpha,\beta)}(x)$ $(\nu=0,1,\dotsc)$
we denote the Jacobi's polynomials,
i.e.\ algebraic polynomials of degree~$\nu$ orthogonal
with the weight function $(1-x)^{\alpha}(1+x)^{\beta}$
on the segment~$[-1,1]$ and normed by the condition
$P_\nu^{(\alpha,\beta)}(1)=1$ $(\nu=0,1,\dotsc)$.

Denote by $a_n(f)$ the Fourier--Jacobi coefficients of a function~$f$,
integrable with the weight function $\Si{x}^2$
on the segment~$[-1,1]$,
with respect to the system of Jacobi polynomials
$\brc{\Px n(x)}_{n=0}^\infty$;
i.e.,
\begin{displaymath}
	a_n(f)=\int_{-1}^1f(x)\Px n(x)\Si{x}^2\,dx
	\quad(n=0,1,\dotsc).
\end{displaymath}

We define the following operators,
having an auxiliary role later on
\begin{align*}
	\T{1;y}{}{f,x}
	  &=\frac1{\pi\Si x}\int_{-1}^1\prn{1-R^2-2\Si y\Si z}f(R)
			\frac{dz}{\sqrt{1-z^2}},\\
	\T{2;y}{}{f,x}
	  &=\frac8{3\pi}\int_{-1}^1\Si{z}^2f(R)
		  \frac{dz}{\sqrt{1-z^2}},
\end{align*}
where $R=xy-z\sqrt{1-x^2}\sqrt{1-y^2}$,
and the corresponding operators of order~$r$
\begin{align*}
	\T{k;y}1{f,x}
	  &=\T{k;y}{}{f,x},\\
	\T{k;\arr y r}r{f,x}
	  &=\T{k;y_r}{}{\T{k;\arr y{r-1}}{r-1}{f,x},x}
	  \quad(r=2,3,\dotsc)
\end{align*}
for $k=1,2$.

\subsection{}

\begin{lmm}\label{lm:bernshtein-markov}
	Let $P_n(x)$ be an algebraic polynomial of degree
	not greater than $n-1$,
	$\allp$, $\alpha>-\frac1p$ and $\rho\ge0$.
	Then the following inequalities hold true
	\begin{gather*}
		\normpar{P'_n(x)}{p,\alpha+\frac12}
		  \le\Cn n\norm{P_n},\\
		\norm{P_n}
		  \le\Cn n^{2\rho}\normpar{P_n}{p,\alpha+\rho},
	\end{gather*}
	where the constants~$\prevC$ and~$\lastC$ do not depend
	on~$n$.
\end{lmm}

Lemma is proved in~\cite{halilova:izv-74}.

\begin{lmm}\label{lm:T*P}
	The operators~$T_{1;y}$ and~$T_{2;y}$ have the following
	properties
	\begin{gather*}
		\T{1;y}{}{\Px\nu,x}=\Px\nu(x)\Py\nu(y),\\
		\T{2;y}{}{\Px\nu,x}=\Px\nu(x)\Px\nu(y)
	\end{gather*}
	for $\nu=0,1,\dotsc$.
\end{lmm}

Lemma~\ref{lm:T*P} is proved in~\cite{potapov:mat-99}.

\begin{lmm}\label{lm:T*fg}
	Let $g(x)\T{k;y}{}{f,x}\in\Lmu$ for every~$y$.
	Then for $k=1,2$ the following equality holds true
	\begin{displaymath}
		\int_{-1}^1f(x)\T{k;y}{}{g,x}\Si{x}^2\,dx
		=\int_{-1}^1g(x)\T{k;y}{}{f,x}\Si{x}^2\,dx.
	\end{displaymath}
\end{lmm}

\begin{proof}
	Let $k=1$ and
	\begin{multline*}
		I_1=\int_{-1}^1f(x)\T{1;y}{}{g,x}\Si{x}^2\,dx\\
		=\frac1\pi\int_{-1}^1\int_{-1}^1f(x)g(R)
			\prn{1-R^2-2\Si y\Si z}\Si x\frac{dz\,dx}{\sqrt{1-z^2}},
	\end{multline*}
	where $R=xy-z\sqrt{1-x^2}\sqrt{1-y^2}$.
	Performing change of variables in the double
	integral by the formulas
	\begin{equation}\label{eq:x-z}
		\begin{split}
		x &=Ry+V\sqrt{1-R^2}\sqrt{1-y^2},\\
		z &=-\frac{R\sqrt{1-y^2}-Vy\sqrt{1-R^2}}
			  {\sqrt{1-\prn{Ry+V\sqrt{1-R^2}\sqrt{1-y^2}}^2}},
		\end{split}
	\end{equation}
	we get
	\begin{multline*}
		I_1=\frac1\pi\int_{-1}^1\int_{-1}^1
			\Si Rf\prn{Ry+V\sqrt{1-R^2}\sqrt{1-y^2}}g(R)\\
		\times\prn{1-\prn{Ry+V\sqrt{1-R^2}\sqrt{1-y^2}}^2-2\Si y\Si V}
				\frac{dV\,dR}{\sqrt{1-V^2}}\\
		=\int_{-1}^1g(R)\T{1;y}{}{f,R}\Si{R}^2\,dR,
	\end{multline*}
	which proves the equality of the lemma for $k=1$.
	
	Let $k=2$ and
	\begin{multline*}
		I_2=\int_{-1}^1f(x)\T{2;y}{}{g,x}\Si{x}^2\,dx\\
		=\frac8{3\pi}\int_{-1}^1\int_{-1}^1f(x)g(R)\Si{x}^2\Si{z}^2
			\frac{dz\,dx}{\sqrt{1-z^2}}.
	\end{multline*}
	Performing change of variables in that double integral
	by the formulas~\eqref{eq:x-z} we get
	\begin{multline*}
		I_2=\frac8{3\pi}\int_{-1}^1\int_{-1}^1
			f\prn{Ry+V\sqrt{1-R^2}\sqrt{1-y^2}}g(R)\Si{R}^2\\
		\times\Si{V}^2\frac{dV\,dR}{\sqrt{1-V^2}}
		  =\int_{-1}^1g(R)\T{2;y}{}{f,R}\Si{R}^2\,dR.
	\end{multline*}
	
	Lemma~\ref{lm:T*fg} is proved.
\end{proof}

\begin{cor}\label{cr:T*fg}
	If $f\in\Lmu$, then for every natural number~$r$ we have
	$\T{k;y}{r}{f,x}\in\Lmu$ $(k=1,2)$.
\end{cor}

\begin{proof}
	Put $g(x)\equiv1$ on $[-1,1]$, considering that by
	Lemma~\ref{lm:T*P}
	(see~\cite[vol.~II, p.~180]{erdelyi-m-o-t:transcendental})
	\begin{gather*}
		\T{1;y}{}{1,x}=\T{1;y}{}{\Px0,x}
		  =\Px0(x)P_2^{(0,0)}(y)=\frac32 y^2-\frac12,\\
		\T{2;y}{}{1,x}=1,
	\end{gather*}
	we have $f(x)\T{k;y}{}{1,x}\in\Lmu$ $(k=1,2)$.
	Hence, applying Lemma~\ref{lm:T*fg}
	we derive
	\begin{displaymath}
		\int_{-1}^1\T{k;y}{}{f,x}\Si{x}^2\,dx
		  =\int_{-1}^1f(x)\T{k;y}{}{1,x}\Si{x}^2\,dx
		\quad (k=1,2).
	\end{displaymath}
	Therefrom it follows that $\T{k;y}{}{f,x}\in\Lmu$.
	Now the corollary is proved by induction.
\end{proof}

\begin{lmm}\label{lm:T*P1}
	Let $f\in\Lmu$.
	For every natural number~$n$
	the following equality holds true
	\begin{displaymath}
		\int_{-1}^1\T{1;y}{}{f,x}P_n^{(1,1)}(y)\,dy
		=\sum_{m=0}^{n-2}a_m(f)\gamma_m(x),
	\end{displaymath}
	where~$\gamma_m(x)$ is an algebraic polynomial
	of degree not greater than $n-2$,
	and $\gamma_m(x)\equiv0$ for $n=0$ or $n=1$.
\end{lmm}

Lemma~\ref{lm:T*P1} is proved in~\cite{potapov:mat-99}.

\begin{lmm}\label{lm:T*-Q}
	Let~$q$ and~$m$ given natural numbers and let $f\in\Lmu$.
	For every natural numbers~$l$ and~$r$ $(l\le r)$
	the function
	\begin{displaymath}
		Q_1^{(l)}(x)
		=\int_0^\pi\dots\int_0^\pi\T{1;\arr{\cos t}l}l{f,x}
			  \prod_{s=1}^r\krn{t_s}\sin^3t_s\,dt_1\dots dt_r
	\end{displaymath}
	is an algebraic polynomial
	of degree not greater than $(q+2)\*(m-1)$.
\end{lmm}

\begin{proof}
	Since
	\begin{displaymath}
		A_s=\krn{t_s}=\sum_{k=0}^{(q+2)(m-1)}a_k\cos kt_s
		=\sum_{k=0}^{(q+2)(m-1)}b_k(\cos t_s)^k,
	\end{displaymath}
	it follows that
	\begin{multline*}
		A_s\sin^2t_s
		  =\sum_{k=0}^{(q+2)(m-1)}b_k(\cos t_s)^k\prn{1-\cos^2t_s}
			=\sum_{k=0}^{(q+2)(m-1)+2}c_k(\cos t_s)^k\\
		=\sum_{k=0}^{(q+2)(m-1)+2}\alpha_k P_k^{(1,1)}(\cos t_s)
		  \quad(s=1,2,\dots,r).
	\end{multline*}
	Hence we have
	\begin{multline*}
		Q_1^{(l)}(x)
		  =\sum_{k=0}^{(q+2)(m-1)+2}\alpha_k\int_0^\pi\dots\int_0^\pi
				  \prod_{\genfrac{}{}{0pt}{1}{s=1}{s\ne l}}^r\krn{t_s}\\
		\times\sin^3t_s\,dt_1\dots dt_{l-1}\,dt_{l+1}\dots dt_r
			  \int_0^\pi\T{1;\arr{\cos t}l}l{f,x}
				P_k^{(1,1)}(\cos t_l)\sin t_l\,dt_l.
	\end{multline*}
	Let
	\begin{displaymath}
		\varphi_{l,k}(x)
		=\int_0^\pi\T{1;\arr{\cos t}l}l{f,x}
			P_k^{(1,1)}(\cos t_l)\sin t_l\,dt_l.
	\end{displaymath}
	Substituting $y=\cos t_l$ we obtain
	\begin{displaymath}
		\varphi_{l,k}(x)
		=\int_{-1}^1\T{1;y}{}{\T{1;\arr{\cos t}{l-1}}{l-1}{f,x},x}
			P_k^{(1,1)}(y)\,dy.
	\end{displaymath}
	Using Lemma~\ref{lm:T*P1} we get
	\begin{displaymath}
		\varphi_{l,k}(x)
		=\sum_{m=0}^{k-2}\gamma_m(x)
		  \int_{-1}^1\T{1;\arr{\cos t}{l-1}}{l-1}{f,R}
			  \Px m(R)\Si{R}^2\,dR.
	\end{displaymath}
	Considering Corollary~\ref{cr:T*fg} we have that
	$\T{1;\arr{\cos t}{l-1}}{l-1}{f,R}\in\Lmu$.
	Applying $l-1$  times Lemma~\ref{lm:T*fg} and
	Lemma~\ref{lm:T*P} we obtain
	\begin{multline*}
		\varphi_{l,k}(x)
		=\sum_{m=0}^{k-2}\gamma_m(x)
			\int_{-1}^1\T{1;\arr{\cos t}{l-2}}{l-2}{f,R}
					\T{1;\cos t_{l-1}}{}{\Px m,R}\\
		\times\Si{R}^2\,dR
		  =\sum_{m=0}^{k-2}\gamma_m(x)\Py m(\cos t_{l-1})\\
		\times\int_{-1}^1\T{1;\arr{\cos t}{l-2}}{l-2}{f,R}
			  \Px m(R)\Si{R}^2\,dR\\
		=\sum_{m=0}^{k-2}\gamma_m(x)
			  \Py m(\cos t_1)\dots\Py m(\cos t_{l-1})\\
		\times\int_{-1}^1f(R)\Px m(R)\Si{R}^2\,dR
			=\sum_{m=0}^{k-2}\gamma_m(x)a_m(f)
				\prod_{s=1}^{l-1}\Py m(\cos t_s),
	\end{multline*}
	where $a_m(f)$ is the Fourier--Jacobi coefficient
	of the function~$f$ with respect
	to the system $\brc{\Px m(x)}_{m=0}^\infty$.
	Substituting $\varphi_{l,k}(x)$ in the
	expression for $Q_1^{(l)}(x)$ we get
	\begin{displaymath}
		Q_1^{(l)}(x)
		=\sum_{k=0}^{(q+2)(m-1)+2}\alpha_k
			\sum_{m=0}^{k-2}\beta_m\gamma_m(x).
	\end{displaymath}
	Since~$\gamma_m(x)$ is an algebraic polynomial of degree
	not greater than $k-2$ for $k\ge2$ and
	$\gamma_m(x)\equiv0$ for $k=0$ and $k=1$,
	then the last equality yields that $Q_1^{(l)}(x)$
	is an algebraic polynomial of degree not greater than
	$(q+2)\*(m-1)$.
	
	Lemma~\ref{lm:T*-Q} is proved.
\end{proof}
	
\begin{lmm}\label{lm:S-Q}
	Let~$q$ and~$m$ given natural numbers.
	Let $f\in\Lmu$.
	For every natural numbers~$l$ and~$r$ $(l\le r)$
	the function
	\begin{displaymath}
		Q_2^{(l)}(x)
		=\int_0^\pi\dots\int_0^\pi\T{2;\arr{\cos t}l}l{f,x}
			  \prod_{s=1}^r\krn{t_s}\sin^5t_s\,dt_1\dots dt_r
	\end{displaymath}
	is an algebraic polynomial of degree not greater
	than $(q+2)\*(m-1)$.
\end{lmm}
	
\begin{proof}
	As shown in Lemma~\ref{lm:T*-Q}
	\begin{multline*}
		A_s=\krn{t_s}
		  =\sum_{k=0}^{(q+2)(m-1)}b_k(\cos t_s)^k\\
		=\sum_{k=0}^{(q+2)(m-1)}\beta_k\Px k(\cos t_s)
		  \quad(s=1,2,\dots,r).
	\end{multline*}
	Hence
	\begin{multline*}
		Q_2^{(l)}(x)
		  =\sum_{k=0}^{(q+2)(m-1)}\beta_k
			\int_0^\pi\dots\int_0^\pi
				\prod_{\substack{s=1\\ s\ne l}}^r\krn{t_s}\\
		\times\sin^5t_s\,dt_1\dots dt_{l-1}\,dt_{l+1}\dots dt_r
			  \int_0^\pi\T{2;\arr{\cos t}l}l{f,x}
				\Px k(\cos t_l)\sin^5 t_l\,dt_l.
	\end{multline*}
	Let
	\begin{multline*}
		\psi_{l,k}(x)
		  =\int_0^\pi\T{2;\arr{\cos t}l}l{f,x}
			\Px k(\cos t_l)\sin^5 t_l\,dt_l\\
		=\int_0^\pi\T{2;\cos t_l}{}{\T{2;\arr{\cos t}{l-1}}{l-1}{f,x},x}
			  \Px k(\cos t_l)\sin^5 t_l\,dt_l
	\end{multline*}
	Substituting $y=\cos t_l$ we obtain
	\begin{displaymath}
		\psi_{l,k}(x)
		=\int_{-1}^1\T{2;y}{}{\T{2;\arr{\cos t}{l-1}}{l-1}{f,x},x}
			  \Px k(y)\Si{y}^2\,dy.
	\end{displaymath}
	Since operator $\T{2;y}{}{f,x}$ is symmetrical on~$x$ and~$y$,
	i.e.\ for every function~$g$
	holds
	$\T{2;y}{}{g,x}=\T{2;x}{}{g,y}$,
	we have
	\begin{displaymath}
		\psi_{l,k}(x)
		=\int_{-1}^1\T{2;x}{}{\T{2;\arr{\cos t}{l-1}}{l-1}{f,y},y}
			\Px k(y)\Si{y}^2\,dy.
	\end{displaymath}
	Since Corollary~\ref{cr:T*fg} yields
	$\T{2;\arr{\cos t}{l-1}}{l-1}{f,y}\in\Lmu$,
	applying Lemma~\ref{lm:T*fg} we obtain
	\begin{displaymath}
		\psi_{l,k}(x)
		=\int_{-1}^1\T{2;\arr{\cos t}{l-1}}{l-1}{f,y}
			\T{2;x}{}{\Px k,y}\Si{y}^2\,dy.
	\end{displaymath}
	Considering the property of the operator~$T_{2;x}$
	from Lemma~\ref{lm:T*P} we get
	\begin{displaymath}
		\psi_{l,k}(x)
		=\Px k(x)\int_{-1}^1\T{2;\arr{\cos t}{l-1}}{l-1}{f,y}
			\Px k(y)\Si{y}^2\,dy.
	\end{displaymath}
	Applying $l-1$ times Lemma~\ref{lm:T*fg} and
	Lemma~\ref{lm:T*P} we obtain
	\begin{multline*}
		\psi_{l,k}(x)
		  =\Px k(x)\Px k(\cos t_1)\dots\Px k(\cos t_{l-1})\\
		\times\int_{-1}^1f(y)\Px k(y)\Si{y}^2\,dy
		  =\Px k(x)a_k(f)\prod_{s=1}^{l-1}\Px k(\cos t_s).
	\end{multline*}
	where $a_k(f)$ is the Fourier--Jacobi coefficient
	of the function~$f$ with respect to the system
	$\brc{\Px k(x)}_{k=0}^\infty$.
	Substituting $\psi_{l,k}(x)$ in the expression
	for $Q_2^{(l)}(x)$ we get
	\begin{displaymath}
		Q_2^{(l)}(x)
		=\sum_{k=0}^{(q+2)(m-1)}\delta_k\Px k(x).
	\end{displaymath}
	Since $\Px k(x)$ is an algebraic polynomial of degree
	not greater than~$k$,
	the last equality implies that $Q_2^{(l)}(x)$ is
	an algebraic polynomial of degree not greater
	than $(q+2)\*(m-1)$.
	
	Lemma is proved.
\end{proof}

\begin{lmm}\label{lm:properties-T}
	Operator~$T_y$ has the following properties
	\begin{enumerate}
		\item The operator $\T y{}{f,x}$ is linear on~$f$;
		\item $\T1{}{f,x}=f(x)$;
		\item $\T y{}{\Px n,x}=\Px n(x)R_n(y)$
		  $(n=0,1,\dotsc)$,\\
		  where $R_n(y)=\Py n(y)+\frac32\Si y\Px n(y)$;
		\item $\T y{}{1,x}=1$;
		\item $a_k\prn{\T y{}{f,x}}=R_k(y)a_k(f)$
		  $(k=0,1,\dotsc)$.
	\end{enumerate}
\end{lmm}

Lemma~\ref{lm:properties-T}
is proved in~\cite{potapov:mat-99}.

\begin{cor}\label{cr:properties-T}
	If $P_n(x)$ is an algebraic polynomial
	of degree not greater than $n-1$,
	then for every natural number~$r$,
	for fixed $y_1,y_2,\allowbreak\dots,y_r$,
	functions $\T{\arr y r}r{P_n,x}$
	and $\Dl{\arr y r}r{P_n,x}$
	are algebraic polynomials on~$x$
	of degree not greater than $n-1$.
\end{cor}

\begin{lmm}\label{lm:elementary}
	If $-1\le x\le1$, $-1\le z\le1$, $0\le t\le\pi$ and
	$R=xy+z\sqrt{1-x^2}\*\sqrt{1-y^2}$,
	then $-1\le R\le1$ and
	\begin{gather*}
		\prn{x\sqrt{1-y^2}+yz\sqrt{1-x^2}}^2\le\Si R,\\
		\prn{\sqrt{1-x^2}y+xz\sqrt{1-y^2}}^2\le\Si R,\\
		\Si x\Si z\le\Si R,\\
		\Si y\Si z\le\Si R.
	\end{gather*}
\end{lmm}

Lemma~\ref{lm:elementary} is proved
in~\cite{potapov:mat-99} and~\cite{potapov:trudy-81}.

\begin{lmm}\label{lm:bound-T}
	Let given numbers~$p$ and~$\alpha$ be such that $\allp$;
	\begin{alignat*}2
		\frac12      &<\alpha\le1
		  &\quad &\text{for $p=1$},\\
		1-\frac1{2p} &<\alpha<\frac32-\frac1{2p}
		  &\quad &\text{for $1<p<\infty$},\\
		1            &\le\alpha<\frac32
		  &\quad &\text{for $p=\infty$}.
	\end{alignat*}
	Let $f\in\Lp$.
	The following inequality holds true
	\begin{displaymath}
		\norm{\T y{}{f,x}}\le C\norm f,
	\end{displaymath}
	where the constant~$C$ does not depend on~$f$ and~$y$.
\end{lmm}

Lemma~\ref{lm:bound-T}
is also proved in~\cite{potapov:mat-99}.

\begin{cor}\label{cr:bound-T}
	Let given numbers~$p$ and~$\alpha$ be such that $\allp$;
	\begin{alignat*}2
		\frac12 		 &<\alpha\le1
		  &\quad &\text{for $p=1$},\\
		1-\frac1{2p} &<\alpha<\frac32-\frac1{2p}
		  &\quad &\text{for $1<p<\infty$},\\
		1            &\le\alpha<\frac32
		  &\quad &\text{for $p=\infty$}.
	\end{alignat*}
	Let $f\in\Lp$.
	The following inequality holds true
	\begin{displaymath}
		\norm{\T{\arr y r}r{f,x}}\le C\norm f,
	\end{displaymath}
	where the constant~$C$ does not depend on~$f$
	and~$y_j$ $(j=1,2,\allowbreak\dots,r)$.
\end{cor}

The corollary is proved
by applying $r$~times Lemma~\ref{lm:bound-T}
taking into consideration Corollary~\ref{cr:T*fg}
(see~\cite{potapov:mat-99}).

\subsection{}

\begin{thm}\label{th:T-Q}
	Let~$q$, $m$ and~$r$ given natural numbers
	and let $f\in\Lmu$.
	The function
	\begin{multline*}
		Q(x)=\frac1{(\gamma_m)^r}
			\int_0^\pi\dots\int_0^\pi
			  \prn{\Dl{\arr t r}r{f,x}-(-1)^r f(x)}\\
		\times\prod_{s=1}^r\krn{t_s}\sin^3t_s\,dt_1\dots dt_r,
	\end{multline*}
	where
	\begin{displaymath}
		\gamma_m=\int_0^\pi\krn t\sin^3t\,dt,
	\end{displaymath}
	is an algebraic polynomial of degree not greater
	than $(q+2)\*(m-1)$.
\end{thm}

\begin{proof}
	To prove the theorem it is sufficient to show that
	for every $l=1,2,\allowbreak\dots,r$ the function
	\begin{displaymath}
		Q^{(l)}(x)
		=\frac1{(\gamma_m)^r}\int_0^\pi\dots\int_0^\pi
			\T{\arr{\cos t}l}l{f,x}
			  \prod_{s=1}^r\krn{t_s}\sin^3t_s\,dt_1\dots dt_r
	\end{displaymath}
	is an algebraic polynomial of degree not greater
	than $(q+2)\*(m-1)$.
	
	It is obvious that the function $Q^{(l)}(x)$
	can be written in the form
	\begin{displaymath}
		Q^{(l)}(x)
		=\frac1{(\gamma_m)^r}
		  \prn{Q_1^{(l)}(x)+\frac32Q_2^{(l)}(x)},
	\end{displaymath}
	where $Q_1^{(l)}(x)$ and $Q_2^{(l)}(x)$ are the functions
	from Lemmas~\ref{lm:T*-Q} and~\ref{lm:S-Q} respectively.
	But, then Lemmas~\ref{lm:T*-Q} and~\ref{lm:S-Q} yield
	that $Q^{(l)}(x)$ is an algebraic polynomial of degree
	not greater than $(q+2)\*(m-1)$.
	
	Theorem is proved.
\end{proof}

\begin{thm}\label{th:HsubE}
	Let given numbers~$p$, $\alpha$, $r$ and~$\lambda$
	be such that $\allp$, $\lambda>0$, $r\in\numN$;
	\begin{alignat*}2
		\alpha &\le2       &\quad &\text{for $p=1$},\\
		\alpha &<3-\frac1p &\quad &\text{for $1<p\le\infty$}.
	\end{alignat*}
	Let $f\in\Lp$ and
	\begin{displaymath}
		\w\le M\delta^\lambda.
	\end{displaymath}
	Then
	\begin{displaymath}
		\E\le CMn^{-\lambda},
	\end{displaymath}
	where the constant~$C$ does not depend on~$f$, $M$
	and~$n$.
\end{thm}

\begin{proof}
	It can easily be proved that under the conditions
	of the theorem, if $f\in\Lp$, then $f\in\Lmu$.
	
	We choose a natural number~$q$ such that $2q>\lambda$,
	and for each natural number~$n$ we choose the natural
	number~$m$ satisfying the condition
	\begin{equation}\label{eq:m}
		\frac{n-1}{q+2}<m\le\frac{n-1}{q+2}+1.
	\end{equation}
	For those~$q$ and~$m$ polynomial $Q(x)$ defined
	in Theorem~\ref{th:T-Q} is an algebraic polynomial
	of degree not greater than $n-1$.
	Hence
	\begin{multline*}
		\E\le\norm{f(x)-(-1)^{r+1}Q(x)}\\
		  =\Bigg\|
			  \frac1{(\gamma_m)^r}
			  \int_0^\pi\dots\int_0^\pi\Dl{\arr t r}r{f,x}\\
		\times\prod_{s=1}^r\krn{t_s}\sin^3t_s\,dt_1\dots dt_r
			\Bigg\|_{p,\alpha}.
	\end{multline*}
	Applying the generalised inequality of Minkowski
	we obtain
	\begin{multline*}
		\E\le\frac1{(\gamma_m)^r}\int_0^\pi\dots\int_0^\pi
			  \norm{\Dl{\arr t r}r{f,x}}\\
		\times\prod_{s=1}^r\krn{t_s}\sin^3t_s\,dt_1\dots dt_r\\
		\le\frac1{(\gamma_m)^r}
			\int_0^\pi\dots\int_0^\pi\wpar r{f,\sum_{j=1}^r t_j}
			  \prod_{s=1}^r\krn{t_s}\sin^3t_s\,dt_1\dots dt_r.
	\end{multline*}
	Hence, considering the conditions of the theorem we have
	(see~\cite[p.~31]{bari:trigonometricheksie})
	\begin{multline*}
		\E\le\frac M{(\gamma_m)^r}
			\int_0^\pi\dots\int_0^\pi
			  \prn{\sum_{j=1}^r t_j}^\lambda\\
		\times\prod_{s=1}^r\krn{t_s}\sin^3t_s\,dt_1\dots dt_r\\
		\le\Cn M\sum_{j=1}^r\frac1{(\gamma_m)^r}
			\int_0^\pi\dots\int_0^\pi t_j^\lambda
			  \prod_{s=1}^r\krn{t_s}\sin^3t_s\,dt_1\dots dt_r.
	\end{multline*}
	Applying the standard evaluation of the Jackson's kernel,
	considering inequality~\eqref{eq:m}, we obtain
	\begin{displaymath}
		\E\le\Cn Mm^{-\lambda}\le\Cn Mn^{-\lambda}.
	\end{displaymath}
	
	Theorem~\ref{th:HsubE} is proved.
\end{proof}

\begin{thm}\label{th:EsubH}
	Let given numbers~$p$, $\alpha$, $r$ and~$\lambda$
	be such that $\allp$, $r\in\numN$, $0<\lambda<2r$;
	\begin{alignat*}2
		\frac12      &<\alpha\le1
		  &\quad &\text{for $p=1$},\\
		1-\frac1{2p} &<\alpha<\frac32-\frac1{2p}
		  &\quad &\text{for $1<p<\infty$},\\
		1            &\le\alpha<\frac32
		  &\quad &\text{for $p=\infty$}.
	\end{alignat*}
	If $f\in\Lp$ and
	\begin{displaymath}
		\E\le\frac M{n^\lambda},
	\end{displaymath}
	then
	\begin{displaymath}
		\w\le CM\delta^\lambda,
	\end{displaymath}
	where the constant~$C$ does not depend on~~$f$, $M$
	and~$\delta$.
\end{thm}

\begin{proof}
	Let $P_n(x)$ be the polynomial of degree not greater
	than $n-1$ such that
	\begin{displaymath}
		\norm{f-P_n}=\E \quad(n=1,2,\dotsc).
	\end{displaymath}
	We construct the polynomials $Q_k(x)$ by
	\begin{displaymath}
		Q_k(x)=P_{2^k}(x)-P_{2^{k-1}}(x) \quad(k=1,2,\dotsc)
	\end{displaymath}
	and $Q_0(x)=P_1(x)$.
	Since for $k\ge1$ we have
	\begin{multline*}
		\norm{Q_k}=\norm{P_2^k-P_{2^{k-1}}}
		  \le\norm{P_{2^k}-f}+\norm{f-P_{2^{k-1}}}\\
		=\Epar{2^k}f+\Epar{2^{k-1}}f,
	\end{multline*}
	then under the conditions of the theorem it follows that
	\begin{displaymath}
		\norm{Q_k}\le\Cn M2^{-k\lambda}.
	\end{displaymath}
	
	It is obvious that without lost in generality we may
	assume that $t_s\ne0$ $(s=1,2,\allowbreak\dots,r)$.
	For $0<|t_s|<\delta$ $(s=1,2,\allowbreak\dots,r)$
	we estimate
	\begin{displaymath}
		I=\norm{\Dl{\arr t r}r{f,x}}.
	\end{displaymath}
	For every natural number~$N$, considering that
	linearity of the operator $\hatT{t_1}{}{f,x}$
	implies the linearity of the
	operator $\hatT{\arr t r}r{f,x}$,
	i.e.\ the linearity of the difference $\Dl{\arr t r}r{f,x}$,
	we have
	\begin{displaymath}
		I\le\norm{\Dl{\arr t r}r{f-P_{2^N},x}}
		  +\norm{\Dl{\arr t r}r{P_{2^N},x}}.
	\end{displaymath}
	Since $P_{2^N}(x)=\sum_{k=0}^N Q_k(x)$,
	we get
	\begin{displaymath}
		I\le\norm{\Dl{\arr t r}r{f-P_{2^N},x}}
		  +\sum_{k=1}^N\norm{\Dl{\arr t r}r{Q_k,x}}.
	\end{displaymath}
	Applying Corollary~\ref{cr:bound-T} we have
	\begin{displaymath}
		I\le\Cn\Epar{2^N}f+\sum_{k=1}^N I_k.
	\end{displaymath}
	
	Let~$N$ be chosen so that
	\begin{equation}\label{eq:delta}
		\frac\pi{2^N}<\delta\le\frac\pi{2^{N-1}}.
	\end{equation}
	We prove that the following inequality holds true
	\begin{equation}\label{eq:Ik}
		I_k\le\Cn M\delta^{2r}2^{k(2r-\lambda)}.
	\end{equation}
	
	Let
	\begin{displaymath}
		\psi_k(x)=\Dl{\arr t r}r{Q_k,x}.
	\end{displaymath}
	It can be proved that
	\begin{multline}\label{eq:psik}
		\psi_k(x)
		  =\frac1{2\pi\Si x}
			\int_0^{t_r}\int_{-u}^u\int_0^\pi
			  \bigg(
				A(v)(R'_v)^2
				\frac{d^2}{dR_v^2}\Dl{\arr t{r-1}}{r-1}{Q_k,R_v}\\
		-(A(v)R_v-2A'(v)R'_v)
				\frac d{dR_v}\Dl{\arr t{r-1}}{r-1}{Q_k,R_v}\\
		+A''(v)\Dl{\arr t{r-1}}{r-1}{Q_k,R_v}
			  \bigg)\,d\varphi\,dv\,du,
	\end{multline}
	where $R_v=x\cos v-\sqrt{1-x^2}\cos\varphi\sin v$,
	\begin{displaymath}
		A(v)=1-R_v^2-2\sin^2v\sin^2\varphi+4\Si x\sin^2v\sin^4\varphi.
	\end{displaymath}
	Applying estimates from Lemma~\ref{lm:elementary}
	and performing change of variables $z=\cos\varphi$
	we obtain
	\begin{displaymath}
		|\psi_k(x)|
		\le\frac{\Cn[cn:psik]}{1-x^2}
		  \int_0^{t_r}\int_{-u}^u\int_{-1}^1B(R_v)
			  \frac{dz}{\sqrt{1-z^2}}\,dv\,du,
	\end{displaymath}
	where
	\begin{multline*}
		B(R_v)=\Si{R_v}^2
			\left|
			  \frac{d^2}{dR_v^2}\Dl{\arr t{r-1}}{r-1}{Q_k,R_v}
			\right|\\
		+\Si{R_v}
			\left|
			  \frac d{dR_v}\Dl{\arr t{r-1}}{r-1}{Q_k,R_v}
			\right|\\
		+\left|\Dl{\arr t{r-1}}{r-1}{Q_k,R_v}\right|
		  =B_1(R_v)+B_2(R_v)+B_3(R_v).
	\end{multline*}
	Therefore using the generalised Minkowski's inequality
	we get
	\begin{equation}\label{eq:sumBmu}
		I_k=\norm{\psi_k(x)}
		  \le\refC{cn:psik}
			  \int_0^{t_r}\int_{-u}^u\int_{-1}^1
				\norm{\frac{B(R_v)}{1-x^2}}
				\frac{dz}{\sqrt{1-z^2}}\,dv\,du.
	\end{equation}
	
	Let $p=1$.
	Considering that $\alpha\le1$ we obtain
	\begin{displaymath}
		I_k
		\le\int_0^{t_r}\int_{-u}^u\int_{-1}^1\int_{-1}^1
			|B(R_v)|\Si{x}^{\alpha-1}\Si{z}^{\alpha-1}
			  \frac{dx\,dz}{\sqrt{1-z^2}}\,dv\,du.
	\end{displaymath}
	
	Let $1<p<\infty$.
	Applying the H\"older's inequality in the inside integral
	in equation~\eqref{eq:sumBmu},
	considering that $\alpha<\frac32-\frac1{2p}$
	we obtain
	\begin{multline*}
		I_k\le\refC{cn:psik}
			\int_0^{t_r}\int_{-u}^u\int_{-1}^1\brc{\int_{-1}^1
			  |B(R_v)|^p\Si{x}^{p(\alpha-1)}
			  \Si{z}^{p(\alpha-1)}\,dx}^{\frac1p}\\
		\times\Si{z}^{-\frac1{2p}}
			  \Si{z}^{-\alpha+\frac1{2p}+\frac12}\,dz\,dv\,du\\
		\le\Cn\int_0^{t_r}\int_{-u}^u
			  \bigg\{
				\int_{-1}^1\int_{-1}^1|B(R_v)|^p
				  \Si{x}^{p(\alpha-1)}\\
		\times\Si{z}^{p(\alpha-1)}\frac{dx\,dz}{\sqrt{1-z^2}}
			  \bigg\}^{\frac1p}\,dv\,du.
	\end{multline*}
	
	Thus, under the conditions of the theorem,
	for $1\le p<\infty$ we have
	\begin{multline*}
		I_k\le\Cn\int_0^{t_r}\int_{-u}^u
			\bigg\{
			  \int_{-1}^1\int_{-1}^1|B(R_v)|^p
				\Si{x}^{p(\alpha-1)}\\
		\times\Si{z}^{p(\alpha-1)}\frac{dx\,dz}{\sqrt{1-z^2}}
			\bigg\}^{\frac1p}\,dv\,du.
	\end{multline*}
	Performing the change of variables in double integral
	by the formulas
	\begin{align*}
		R&=x\cos v-z\sqrt{1-x^2}\sin v,\\
		V&=\frac{x\sin v+z\sqrt{1-x^2}\cos v}
			{\sqrt{1-\prn{x\cos v-z\sqrt{1-x^2}\sin v}^2}},
	\end{align*}
	we obtain
	\begin{multline*}
		I_k\le\lastC\int_0^{t_r}\int_{-u}^u
			  \bigg\{
				\int_{-1}^1\int_{-1}^1|B(R)|^p
				  \Si{R}^{p(\alpha-1)}\\
		\times\Si{V}^{p(\alpha-1)-\frac12}\,dR\,dV
			  \bigg\}^{\frac1p}\,dv\,du.
	\end{multline*}
	Since, under the conditions of theorem $\alpha>1-\frac1{2p}$,
	it follows that
	\begin{multline*}
		I_k\le\Cn
			\int_0^{t_r}\int_{-u}^u
			  \brc{
				\int_{-1}^1|B(R)|^p\Si{R}^{p(\alpha-1)}\,dR
			  }^{\frac1p}\,dv\,du\\
		\le\Cn t_r^2\normpar{B(R)}{p,\alpha-1}.
	\end{multline*}
	
	Let now $p=\infty$.
	Considering the estimates from Lemma~\ref{lm:elementary}
	and that $\alpha\ge1$,
	inequality~\eqref{eq:sumBmu} yields
	\begin{multline*}
		I_k\le\refC{cn:psik}
			\int_0^{t_r}\int_{-u}^u\int_{-1}^1
			  \esssup_{-1\le x\le1}|B(R_v)|
				\Si{x}^{\alpha-1}\frac{dz}{\sqrt{1-z^2}}\,dv\,du\\
		\le\refC{cn:psik}\normpar{B(x)}{\infty,\alpha-1}
			  \int_0^{t_r}\int_{-u}^u\int_{-1}^1
				\Si{z}^{-\alpha+\frac12}\,dz\,dv\,du.
	\end{multline*}
	Hence, considering that $\alpha<\frac32$ we get
	\begin{displaymath}
		I_k\le\Cn t_r^2\normpar{B(x)}{\infty,\alpha-1}.
	\end{displaymath}
	Thus for all $\allp$ we proved that
	\begin{displaymath}
		I_k\le\Cn t_r^2\normpar{B(x)}{p,\alpha-1}.
	\end{displaymath}
	
	Applying Lemma~\ref{lm:bernshtein-markov} and
	Corollaries~\ref{cr:properties-T} and~\ref{cr:bound-T}
	under the conditions of the theorem we obtain
	\begin{multline*}
		I_k=\norm{\Dl{\arr t r}r{Q_k,x}}
		  \le\lastC t_r^2\normpar{B(x)}{p,\alpha-1}\\
		\le\lastC t_r^2
			\prn{\normpar{B_1(x)}{p,\alpha-1}
			  +\normpar{B_2(x)}{p,\alpha-1}
			  +\normpar{B_3(x)}{p,\alpha-1}}\\
		=\lastC t_r^2
			\bigg\{
			  \normpar{\frac{d^2}{dx^2}\Dl{\arr t{r-1}}{r-1}{Q_k,x}}
				{p,\alpha+1}\\
		+\normpar{\frac d{dx}\Dl{\arr t{r-1}}{r-1}{Q_k,x}}
				{p,\alpha}
			  +\normpar{\Dl{\arr t{r-1}}{r-1}{Q_k,x}}{p,\alpha-1}
			\bigg\}\\
		\le\Cn t_r^2 2^{2k}\norm{\Dl{\arr t{r-1}}{r-1}{Q_k,x}}.
	\end{multline*}
	Applying $r$~times this inequality it follows that
	\begin{displaymath}
		I_k\le\Cn t_1^2\dots t_r^2 2^{2kr}\norm{Q_k}.
	\end{displaymath}
	Therefore we have
	\begin{displaymath}
		I_k\le\Cn M\delta^{2r}2^{k(2r-\lambda)}.
	\end{displaymath}
	
	Inequality~\eqref{eq:Ik} is proved.
	
	Inequalities~\eqref{eq:Ik} and~\eqref{eq:delta} yield
	\begin{displaymath}
		I\le\Cn M\bigg(
			\delta^\lambda+\delta^{2r}\sum_{k=1}^N2^{k(2r-\lambda)}
		  \bigg)
		\le\Cn M\prn{\delta^\lambda+\delta^{2r}2^{N(2r-\lambda)}}
		\le\Cn M\delta^\lambda.
	\end{displaymath}
	
	Theorem~\ref{th:EsubH} is completed.
\end{proof}

Now we formulate the theorem of coincidence
of the class $H(p,\alpha,r,\lambda)$
with the class $E(p,\alpha,\lambda)$,
and the inverse theorem.

\begin{thm}\label{th:coincidence}
	Let given numbers~$p$, $\alpha$, $r$ and~$\lambda$
	be such that $\allp$, $0<\lambda<2r$, $r\in\numN$;
	\begin{alignat*}2
		\frac12      &<\alpha\le1
		  &\quad &\text{for $p=1$},\\
		1-\frac1{2p} &<\alpha<\frac32-\frac1{2p}
		  &\quad &\text{for $1<p<\infty$},\\
		1            &\le\alpha<\frac32
		  &\quad &\text{for $p=\infty$}.
	\end{alignat*}
	The class $H(p,\alpha,r,\lambda)$
	coincides with the class~$E(p,\alpha,\lambda)$.
\end{thm}

Theorem~\ref{th:coincidence}
is implied by Theorems~\ref{th:HsubE} and~\ref{th:EsubH}
proved above.

\begin{thm}\label{th:converse}
	Let given numbers~$p$, $\alpha$, $r$ and~$\lambda$
	be such that $\allp$, $0<\lambda<2r$, $r\in\numN$;
	\begin{alignat*}2
		\frac12      &<\alpha\le1
		  &\quad &\text{for $p=1$},\\
		1-\frac1{2p} &<\alpha<\frac32-\frac1{2p}
		  &\quad &\text{for $1<p<\infty$},\\
		1            &\le\alpha<\frac32
		  &\quad &\text{for $p=\infty$}.
	\end{alignat*}
	If $f\in\Lp$, then the following inequality holds
	\begin{displaymath}
		\wpar r{f,\frac1n}
		\le\frac\Cn{n^{2r}}\sum_{\nu=1}^n\nu^{2r-1}\Epar\nu f,
	\end{displaymath}
	where the constant~$C$ does not depend on~$f$ and~$n$.
\end{thm}

\begin{proof}
	Let $P_n(x)$ be the polynomial of degree not greater
	than $n-1$ such that
	\begin{displaymath}
		\norm{f-P_n}=\E \quad(n=1,2,\dotsc),
	\end{displaymath}
	and
	\begin{displaymath}
		Q_k(x)=P_{2^k}(x)-P_{2^{k-1}}(x) \quad(k=1,2,\dotsc),
	\end{displaymath}
	$Q_0(x)=P_1(x)$.
	
	For given~$n$ we chose the natural number~$N$ such that
	\begin{displaymath}
		\frac n2<2^N\le n+1.
	\end{displaymath}
	By the proof of Theorem~\ref{th:EsubH} it follows that
	\begin{multline*}
		\wpar r{f,\frac1n}
		  \le\Cn
			\bigg(
			  \Epar{2^N}f
			  +\frac1{n^{2r}}\sum_{\mu=1}^N2^{2\mu r}\norm{Q_k}
			\bigg)\\
		\le2\lastC
			\bigg(
			  \Epar{2^N}f
			  +\frac1{n^{2r}}
				  \sum_{\mu=1}^N2^{2\mu r}
					\prn{\Epar{2^\mu}f+\Epar{2^{\mu-1}}f}
			  \bigg)\\
		\le4\lastC
			\bigg(
			  \Epar{2^N}f
			  +\frac1{n^{2r}}
				\sum_{\mu=0}^{N-1}2^{2(\mu+1)r}\Epar{2^\mu}f
			\bigg)\\
		\le\frac\Cn{n^{2r}}
			\sum_{\mu=0}^N2^{2(\mu+1)r}\Epar{2^\mu}f.
	\end{multline*}
	Considering that for $\mu\ge1$ we have
	\begin{displaymath}
		\sum_{\nu=2^{\mu-1}}^{2^\mu-1}\nu^{2r-1}\Epar\nu f
		\ge\Epar{2^\mu}f2^{\mu-1}2^{(\mu-1)(2r-1)}
		\ge\Cn2^{2(\mu+1)r}\Epar{2^\mu}f,
	\end{displaymath}
	it follows that
	\begin{multline*}
		\wpar r{f,\frac1n}
		  \le\frac\Cn{n^{2r}}
			\bigg(
			  2^{2r}\Epar1f
				+\sum_{\mu=1}^N\sum_{\nu=2^{\mu-1}}^{2^\mu-1}
				  \nu^{2r-1}\Epar\nu f
			  \bigg)\\
		\le\frac\Cn{n^{2r}}\sum_{\nu=1}^n\nu^{2r-1}\Epar\nu f.
	\end{multline*}
	
	Theorem~\ref{th:converse} is proved.
\end{proof}

\bibliographystyle{amsplain}
\bibliography{maths}

\providecommand\cprime{'}
\providecommand{\bysame}{\leavevmode\hbox to3em{\hrulefill}\thinspace}
\providecommand{\MR}{\relax\ifhmode\unskip\space\fi MR }
\providecommand{\MRhref}[2]{%
  \href{http://www.ams.org/mathscinet-getitem?mr=#1}{#2}
}
\providecommand{\href}[2]{#2}
\begin{thebibliography}{1}

\bibitem{erdelyi-m-o-t:transcendental}
A.~Erd{\'e}lyi, W.~Magnus, F.~Oberhettinger, and F.~G. Tricomi, \emph{Higher
  transcendental functions}, Three volumes, Robert E. Krieger Publishing Co.
  Inc., Melbourne, Fla., 1981, (Russian translation, Gosudarstv. Izdat.
  Inostranno\u{\i} Literatury, Moscow, 1969). \MR{84h:33001}

\bibitem{halilova:izv-74}
B.~A. Halilova, \emph{O nekotorykh otsenkakh dlya polinomov}, Izv. Akad. Nauk
  Azerba\u{\i}dzhan. SSR Ser. Fiz.-Tekhn. Mat. Nauk (1974), no.~2, 46--55.
  \MR{50 \#4863}

\bibitem{potapov:trudy-81}
M.~K. Potapov, \emph{Ob usloviyakh sovpadeniya nekotorykh klassov
  funktsi\u{\i}}, Trudy Sem. Petrovsk. (1981), no.~6, 223--238. \MR{82i:46053}

\bibitem{potapov:mat-99}
\bysame, \emph{O sovpadenii klassov funktsi\u{\i} opredelyaemykh operatorom
  obobshchennogo sdviga ili poryadkom nailuchshego priblizhenya
  algebraicheskimi mnogochlenami}, Mat. Zametki \textbf{66} (1999), no.~2,
  242--257. \MR{2000k:41008}

\end{thebibliography}

\end{document}